\documentclass[11pt,titlepage]{article}
\usepackage{amsmath,amssymb,graphicx,color}

\usepackage{setspace}
\title{\huge Optimal Multivalued Shattering}
\author{Zolt\'an F\"uredi\footnotemark[1]\footnotemark[2] \and Attila Sali\footnotemark[3]\footnotemark[4]}
\footnotetext[1]{Department of Mathematics University of Illinois at Urbana-Champaign
1409 W. Green Street
Urbana, Illinois 61801-2975, USA. E-mail: {\tt z-furedi@illinois.edu} and\\
Alfr\'ed R\'enyi Institute of Mathematics
Hungarian Academy of Sciences
Budapest, P.O.Box 127 H-1364 Hungary.}
\footnotetext[2]{
Research supported in part by the Hungarian National Science Foundation
 OTKA,  by the National Science Foundation under grant NFS DMS 09-01276,
  and by the European Research Council Advanced Investigators Grant 267195.
}
\footnotetext[3]{Alfr\'ed R\'enyi Instiute of Mathematics
Hungarian Academy of Sciences
Budapest, P.O.Box 127 H-1364 Hungary. E-mail: {\tt sali@renyi.hu}}
\footnotetext[4]{Partially supported by Hungarian National Science Foundation
  OTKA grant No. NK 78439.
\newline
This copy was printed on \today, 
   \quad    {\rm\small {\jobname}.tex}
 \newline ${}$
   \hfill
Version as of
    September 7, 2011.
}

\date{}

\newcommand{\keywords}[1]{\newline\textbf{Keywords:} #1}
\newtheorem{defn}{Definition}
\newtheorem{thm}[defn]{Theorem}

\newtheorem{prop}[defn]{Proposition}
\newenvironment{proof}{\medskip                    
\noindent{\scshape Proof:}}{\hfill $\Box$\medskip\par}  
\newcommand{\forb}{\mathsf{forb}}

\begin{document}
\maketitle
\begin{abstract}
We have found the most general extension of the celebrated
Sauer, Perles and Shelah, Vapnik and Chervonenkis result
 from 0-1 sequences to $k$-ary codes still giving a polynomial bound.

Let $\mathcal{C}\subseteq \{ 0,1,\dots, k-1 \}^n$ be a $k$-ary code of length $n$.
For a subset of coordinates $S\subset \{1,2,\ldots ,n\}$
  the projection of $\mathcal{C}$ to $S$ is denoted by $\mathcal{C}\vert_S$.
We say that $\mathcal{C}$ $(i,j)$-{\em shatters} $S$ if $\mathcal{C}\vert_S$ contains
 all the $2^{|S|}$ distinct vectors (codewords) with coordinates $i$ and $j$.
Suppose that $\mathcal{C}$ does not $(i,j)$-shatter any coordinate set of
size $s_{i,j}\geq 1$ for every $1\leq i< j\leq q$ and let $p=\sum (s_{i,j}-1)$.
Using a natural induction we prove that
 $$ |{\mathcal C}|\leq O(n^p)
   $$
 for any given $p$ as $n\to \infty$ and give a construction showing that
  this exponent is the best possible.

 Several open problems are mentioned.\\
    \\
\end{abstract}

\section{Introduction}
Let $[n]$ denote the set $\{1,2,\ldots ,n\}$ while let $(k)$ denote $\{0,1,\ldots ,k-1\}$
 and for any set $S$, let $2^S$ denote the family of all $2^{|S|}$subsets of $S$ and let
${S\choose k}$  denote all ${|S|\choose k}$
subsets of $S$ of size $k$.
Consider a family
$\mathcal{F}$ of subsets of $[n]$. We say that $\mathcal{F}$ {\it shatters} $S$ if
\begin{equation*}\{E\cap S\,:\,E\in\mathcal{F}\}=2^{S}.
   \end{equation*}
The following  result has a variety of  applications
 including learning theory and applied probability.
\begin{thm}\label{thm:sauer}[{\rm Sauer{\cite{Sa}}, Perles, Shelah{\cite{Sh}},
Vapnik, Chervonenkis{\cite{VC}}}]\enskip
Let $\mathcal{F}$ be a family of subsets of $[n]$ with no shattered
set of size $s$. Then
\begin{equation}\label{eq:sauer}
  |\mathcal{F}|\le {n\choose s-1}+{n\choose s-2}+\cdots +{n\choose 0}
   \end{equation}
and 
 this bound is the best possible.\end{thm}
Karpovsky and Milman~\cite{km} and independently Steele~\cite{st} gave a multivalued generalization
 of the result above.
Let $\mathcal{C}\subseteq (k)^n$ be a set of codewords (vectors).
A codeword $\mathbf{c}$ can also be viewed as a \emph{function} from $[n]$ to $(k)$.
The code $\mathcal{C}$ is said to shatter $S\subseteq [n]$ if
\begin{equation*} 
  \left\{\mathbf{c}\vert_S\colon \mathbf{c}\in\mathcal{C}\right\}= (k)^S,
  \end{equation*}
   the set of all functions from $S$ to $(k)$.
\begin{thm}\label{thm:multishatter}
[{\rm Karpovsky and Milman~\cite{km} and independently Steele~\cite{st}
(see also Frankl~\cite{frankl}, Alon \cite{al}, Anstee \cite{an})}]\enskip
 Let $1\le s\le n$ be an integer and let $\mathcal{C}\subseteq (k)^n$ be
 a set of codewords with no shattered set of size $s$. Then
\begin{equation}\label{eq:karpmill}
  |\mathcal{C}|\le \sum_{i=0}^{s-1}(k-1)^{n-i}\binom{n}{i}.
  \end{equation}
\end{thm}
An important difference between the bounds is that (\ref{eq:sauer}) is polynomial in $n$ (for fixed $s$),
but (\ref{eq:karpmill}) is exponential.
The same phenomenon happens when uniform set systems are considered.
The uniform version of Theorem~\ref{thm:sauer} was proven by Frankl and Pach \cite{FP}
 (for a strengthening and algebraic connections see Anstee et.al. \cite{ars}).
\begin{thm}\label{thm:fp}[{\rm  Frankl and Pach \cite{FP}}]\enskip
Let $n,d,s$ be positive integers such that $d\le n$ and $s\le n/2$.
Let $\mathcal{F}\subseteq\binom{[n]}{d}$ be a $d$-uniform set system that does not shatter
 an $s$-element set, then
\begin{equation*}
  |\mathcal{F}|\le\binom{n}{s-1}.
  \end{equation*}
  \end{thm}
Recently, Heged\H{u}s and R\'onyai \cite{hr} gave two multivalued generalizations.
\begin{thm}\label{thm:hr-unif} [{\rm Heged\H{u}s and  R\'onyai \cite{hr}}]\enskip
 Let $0\le d\le (k-1)n$ and $s-1\le n/2$.
 Let $\mathcal{C}\subseteq (k)^n$ be a code with no shattered set of size $s$
 and suppose that $\sum_{i=1}^nc_i=d$ for every $\mathbf{c}\in\mathcal{C}$.
  Then
\begin{equation*}
  |\mathcal{C}|\le \sum_{i=0}^{s-1}(k-1)^{n-i}\left(\binom{n}{i}-\binom{n}{i-1}\right).
  \end{equation*}
\end{thm}
Note that this bound is exponential in $n$.
\begin{thm}\label{thm:hr-hamming}  [{\rm Heged\H{u}s and  R\'onyai \cite{hr}}]\enskip
Let $0\le d\le n$ and $0\le d+s\le n+1$.
Let $\mathcal{C}\subseteq (k)^n$ be a code with no shattered set of size $s$
 and suppose that $|\{i\in [n]\colon c_i\ne 0\}|=d$ for every $\mathbf{c}\in\mathcal{C}$.
 Then
\begin{equation*}
  |\mathcal{C}|\le \binom{n}{s-1}\sum_{i=0}^{d}(k-2)^{i}\binom{n-s+1}{i}.
  \end{equation*}
\end{thm}
One cannot expect an exponential bound here 
 since the total number of codewords with support of size $d$ is polynomial.
     \par
A code $\mathcal C$ and the corresponding matrix $\mathcal{M}$ formed by the
 codewords are called {\it reverse-free} if $\mathcal M$ does not have a
 submatrix of the form $\left(\begin{array}{cc}a&b\\b&a\end{array}\right)$ for any
 distinct $a$ and $b$.
How large a reverse-free code ${\mathcal C}\subset (k)^n$ can be?
It was proved in~\cite{fkms} that
\begin{equation}\label{eq:fkms}
\max |{\mathcal C}|=\Theta\left(n^{\binom{k}{2}}\right).
\end{equation}

This can lead to the following version of {\it multivalued shattering}.
Let $\mathcal{C}\subseteq (k)^n$ be a set of codewords.
$\mathcal{C}$ $(i,j)$-shatters $S\subseteq [n]$ if $\mathcal{C}\vert_S$ contains all $2^{|S|}$
  functions from $S$ to $\{i,j\}$. Let $k\ge 2$ be a fixed integer,
$\vec{s}=(s_{0,1},s_{0,2},\ldots s_{k-2,k-1})$ be a
positive integer vector of length $\binom{k}{2}$ whose entries are indexed by
 ordered pairs $(i,j)$ with $0\leq i<j\leq k-1$.

The main result of the present paper is the following theorem.
\begin{thm}\label{thm:ij-shatter}
Suppose that $\mathcal{C}\subset (k)^n$ does not $(i,j)$-shatter any coordinate set of
 size $s_{i,j}\geq 1$ for every $0\leq i< j\leq k-1$.
Then
\begin{equation}\label{eq:ij-shatter}
 |\mathcal{C}| \leq \sum_{0\le \alpha_{i,j}\le s_{i,j}-1}
 \binom{n}{\alpha_{0,1},\alpha_{0,2},\ldots ,\alpha_{k-2,k-1},n-\sum_{0\le i<j\le k-1}\alpha_{i,j}}
 ={\displaystyle O\left(n^{p}\right),}
\end{equation}
  where the sum is taken for all possible choices of $\alpha_{i,j}$'s and
  $p=\sum_{0\le i<j\le k-1}(s_{i,j}-1)$.

On the other hand, when $p$ is fixed and $n\to \infty$ then there exist codes
$\mathcal{C}\subset (k)^n$ such that they do not $(i,j)$-shatter any coordinate set of
 size $s_{i,j}\geq 1$ for every $0\leq i< j\leq k-1$ and
\begin{equation}\label{eq:lowerbound}
|\mathcal{C}|=\Omega\left(n^{p}\right).
\end{equation}
\end{thm}
In other words, if  $\forb (n, \vec{s})$ denotes the maximum
number of codewords of  a code $\mathcal{C}$ of length $n$ over the alphabet
$(k)$ such that $\mathcal{C}$ does not $(i,j)$-shatter any coordinate set of
size $s_{i,j}$ then
\begin{equation*}
  \forb (n, \vec{s})=\Theta(n^p).
   \end{equation*}

\section{A hierarchy of Vapnik-Chervonenkis type dimensions}
The \emph{{\rm VC}-dimension} of a set system $\mathcal{F}\subseteq 2^{[n]}$ is the maximum $d$
 that $ \mathcal{F}$ shatters a set of size $d$.
Theorem~\ref{thm:sauer} bounds the size of a set system whose VC-dimension is less than $s$.
Vapnik and Chervonenkis used it for bounds on the sample size necessary to obtain uniformly
 good empirical estimates for the expectations of all random variables of a given class.
Since then it has found applications in learning theory, such as concepts with bounded
 VC-dimensions are effectively learnable. \par
Theorem~\ref{thm:multishatter} allows the definition of another dimension,
 \emph{{\rm KM}-dimension} of codes (systems of multisets) as follows.
The KM-dimension of  $\mathcal{C}\subseteq (k)^n$ is the maximum $d$ that $\mathcal{C}$
 shatters a set of size $d$.
Theorem~\ref{thm:multishatter} gives a bound on the size of a code of KM-dimension less than $s$.
However, this bound is exponential function of $n$.\par
Haussler and Long \cite{hl} introduced other generalizations of VC-dimension,
 motivated by statistical applications. The \emph{G-dimension} of $\mathcal{C}\subseteq (k)^n$
 is the maximum $d$ that there exists a vector $\vec{y}=(y_1,y_2,\ldots y_d)\in (k)^d$ and a
 subset $D=\{i_1,i_2,\ldots i_d\}\subseteq [n]$ such that for all subsets $I\subseteq D$
 there exists $\mathbf{c}=(c_1,c_2,\ldots c_n)\in\mathcal{C}$  such that $c_{i_j}=y_j$ for
 $i_j\in I$ and $c_{i_t}\ne y_t$ for  $i_t\not\in I$.

The \emph{\rm{P}-dimension} of $\mathcal{C}$ is the maximum $d$ that there exists a vector
 $\vec{y}$ and a subset $|D|=d$ of $[n]$ such that for all subsets $I\subseteq D$ there exists
 $\mathbf{c}\in\mathcal{C}$ such that $c_{i_j}\ge y_j$ for $i_j\in I$ and $c_{i_t}< y_t$ for
 $i_t\not\in I$.

The \emph{{\rm GP}-dimension} of $\mathcal{C}$ is the maximum $d$ that there exists a vector
 $\vec{y}$ and a subset $|D|=d$ of $[n]$ such that for all subsets $I\subseteq D$ there exists
 $\mathbf{c}\in\mathcal{C}$ such that $c_{i_j}=y_j$ for $i_j\in I$ and $c_{i_t}< y_t$ for  $i_t\not\in I$.

Finally, the \emph{{\rm N}-dimension} (or Natarajan-dimension~\cite{Na}) of $\mathcal{C}$ is
 the maximum $d$ that there exist  vectors $\vec{y}$ and $\vec{z}$ with $z_i<y_i\colon i=1,2,\ldots d$
 and a subset $|D|=d$ of $[n]$ such that for all subsets
 $I\subseteq D$ there exists $\mathbf{c}\in\mathcal{C}$ such that $c_{i_j}=y_j$ for $i_j\in I$ and
 $c_{i_t}=z_t$ for  $i_t\not\in I$.\par
It is easy to see that each of the above dimensions coincide with the VC-dimension in the case of $k=1$.
We also have
\begin{equation}\label{eq:dimensions}
\dim_{\rm KM}(\mathcal{C})\le \dim_{\rm N}(\mathcal{C})\le \dim_{\rm GP}(\mathcal{C})
   \le\left\{
\begin{array}{c} \dim_{\rm G}(\mathcal{C})\\ \dim_{\rm P}(\mathcal{C})\end{array}\right.
\end{equation}
The concept of $(i,j)$-shattering allows us to define a new dimension which is between
 KM-dimension and N-dimension.

The \emph{bi-dimension} of $\mathcal{C}\subseteq (k)^n$ is the maximum $d$ that there exist
 $i<j\in (k)$ and and a set $D\subseteq [n]$ of size $d$ that  $\mathcal{C}$   $(i,j)$-shatters $D$.
If a set $D$ is KM-shattered by $\mathcal{C}$, then $\mathcal{C}\vert_D$ is the set of all functions
 from  $D$ to $(k)$, in particular it contains all functions from $D$ to $\{i,j\}$ for any pair
 $i<j\in (k)$, so $D$ is $(i,j)$-shattered by $\mathcal{C}$.
This shows $$\dim_{\rm KM}(\mathcal{C})\leq \dim_{\rm bi}(\mathcal{C}).$$
 On the other hand, if $D$ is $(i,j)$-shattered by $\mathcal{C}$,
 then $D$ satisfies the condition of N-dimension with vectors $\vec{z}=(i,i,\ldots ,i)$ and
 $\vec{y}=(j,j,\ldots ,j)$, so the N-dimension of $\mathcal{C}$ is at least as large as its bi-dimension.
\par
Let $\mathcal{M}_{\rm X}(n,s)$ denote the maximum size of a code of length $n$
 and $\rm X$-dimension not exceeding $s$ (${\rm X}\in\{{\rm KM, bi, N,GP, G,P}\}$).
Then (\ref{eq:dimensions}) and the observations above imply
\begin{equation*}
\left.
\begin{array}{c}\mathcal{M}_{\rm G}(n,s) \\\mathcal{M}_{\rm P}(n,s)\end{array}\right\}\le
\mathcal{M}_{\rm GP}(n,s)\le \mathcal{M}_{\rm N}(n,s)\le \mathcal{M}_{\rm bi}(n,s)\le
\mathcal{M}_{\rm KM}(n,s).
\end{equation*}
In fact, Haussler and Long \cite{hl} proved that
$$\mathcal{M}_{\rm G}(n,s)=\mathcal{M}_{\rm P}(n,s)=\mathcal{M}_{\rm GP}(n,s)
    = \sum_{0\leq i\leq s}{n\choose i}(k-1)^i.
$$
  $$\mathcal{M}_{\rm N}(n,s)\leq \sum_{0\leq i\leq s}{n\choose i}{k \choose 2}^i.$$
These bounds are polynomial in $n$.
Theorem~\ref{thm:ij-shatter} implies that $\mathcal{M}_{\rm bi}(n,s)$ is polynomial, as well,
 since $\mathcal{M}_{\rm bi}(n,s)=\forb(n,\vec{s})$ for the vector $\vec{s}$ whose
 coordinates are all $s+1$.
However, $\mathcal{M}_{\rm KM}(n,s)$ is exponential according to Theorem~\ref{thm:multishatter}.
An extremal property of bi-dimension 
  is that it is the weakest restriction that
still results in polynomial bound. Indeed, if there is a pair of symbols $i,j$
such that there is no restriction involving only that pair, then one can
select all codewords  $\mathcal{C}=\{i,j\}^n$ so that $\mathcal{C}$ does not
violate any restrictions yet it is of exponential size.

\section{Proofs}
In this section we give two versions of the proof of the upper bound in Theorem~\ref{thm:ij-shatter}.
The lower bound (\ref{eq:lowerbound}) follows from Proposition~\ref{prop:lowerbound}. \newline
\paragraph{Branching proof.}\enskip
Let $\mathcal{C}\subset (k)^n$ be a code avoiding an $(i,j)$-shattered set of size $s_{i,j}$
 for all $0\leq i< j\leq k-1$.
 The following branching process will be applied to $\mathcal{C}$ successively $n$ times.

Let $\mathcal{B}$ be a set of codewords of length $t\ge 1$ over alphabet $(k)$.
Let  $\mathcal{B}_0$ denote the set of suffices of length $t-1$ of codewords in $\mathcal{B}$.
Note, that if $t=1$, then  $\mathcal{B}_0$ has one element, the \emph{empty string}.
If a codeword $\mathbf{b}\in \mathcal{B}_0$ appears with more than one first coordinate in
 $\mathcal{B}$, say with $i_1 < i_2 < \ldots < i_w$, then $\mathbf{b}$ will be put into the $(w-1)$
 sets $\mathcal{B}_{i_1,i_2}, \mathcal{B}_{i_1,i_3},\ldots ,\mathcal{B}_{i_1,i_w}$.
We get
\begin{equation*}
|\mathcal{B}|=|\mathcal{B}_0|+\sum_{0\le i<j\le k-1} |\mathcal{B}_{i,j}|.
\end{equation*}
$\mathcal{B}_{i,j}$ is said to be obtained by $(i,j)$-\emph{branching at step} $t$ from
 $\mathcal{B}$.
\par
Thus, the process starts with $\mathcal{B} = \mathcal{C}$ and $t=n$, and continues with
 $t=n-1,n-2, \ldots ,1$.
At step $t$ every set of codewords obtained at step $t+1$ is branched.
At the end, there are $|\mathcal{C}|$ singleton sets each containing the empty string.
For an example see Figure~\ref{fig:branching}.
\begin{figure}
 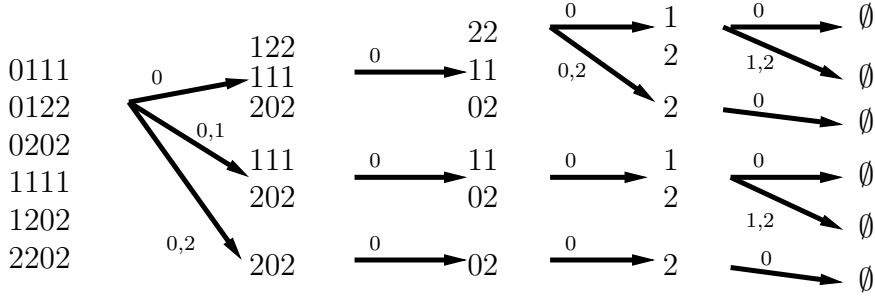
\caption{Branching example}\label{fig:branching}
\end{figure}
Every singleton set is a result of a series of branchings, say $\alpha_{i,j}$ $(i,j)$-branchings
 for $0\le i<j\le k-1$.
If $\alpha_{i,j}\ge s_{i,j}$ for some pair $i,j$, and these branchings occur at steps
 $t_1,t_2,\ldots t_{\alpha_{i,j}}$, then $\mathcal{C}$ $(i,j)$-shatters the set
 $\{t_1,t_2,\ldots t_{\alpha_{i,j}}\}$ that contradicts the assumptions.
The maximum possible number of singleton sets with $\alpha_{i,j}$ $(i,j)$-branchings is equal to
 the number of $n$-permutations of $\alpha_{i,j}$ objects of type $(i,j)$ for $0\le i<j\le k-1$
 and $n-{\displaystyle \sum_{0\le i<j\le k-1}\alpha_{i,j}}$ objects of ``no branching'' type,
 which is exactly the multinomial coefficient
  $$\binom{n}{\alpha_{0,1},\alpha_{0,2},\ldots ,\alpha_{k-2,k-1},n-\sum_{0\le i<j\le k-1}\alpha_{i,j}}.
   $$
This provides the upper bound (\ref{eq:ij-shatter}). \hfill $\square$
   \paragraph{Induction proof.}\enskip
Let $\mathcal{C}_{i,j}^i\subseteq \mathcal{C}$ consist of those codewords $\mathbf{c}$ that $c_n=i$
 and there exists a codeword $\mathbf{c}'\in\mathcal{C}$ that only differs from  $\mathbf{c}$
 in the last coordinate and  $c'_n=j$.
$\mathcal{C}_{i,j}^j\subseteq \mathcal{C}$ is defined similarly.
If $s_{i,j}=1$, then both  $\mathcal{C}_{i,j}^i$ and $\mathcal{C}_{i,j}^j$ are
empty.
Otherwise, let $\vec{s}_{i,j}$ be the vector obtained from $\vec{s}$ by decreasing the
$(i,j)$\textsuperscript{th} coordinate by one.
Then obviously $|\mathcal{C}_{i,j}^i|=|\mathcal{C}_{i,j}^j|\le \forb(n-1,\vec{s}_{i,j})$.
Let $\mathcal{C}\vert_{[n-1]}=\{\mathbf{c}\vert_{[n-1]}\colon \mathbf{c}\in\mathcal{C}\}$
 be the set of length $n-1$ prefixes of codewords in  $\mathcal{C}$.
Clearly, $\mathcal{C}\vert_{[n-1]}\le \forb(n-1,\vec{s})$. On the other hand,
\begin{equation}\label{eq:induction}
| \mathcal{C}|\le |\mathcal{C}\vert_{[n-1]}|+\sum_{0\le i<j\le k-1}|\mathcal{C}_{i,j}^j|.
\end{equation}
In order to prove (\ref{eq:ij-shatter})
 using induction we have to give upper bound for $\forb(1, \vec{s})$.
In this case $i$ and $j$ both can be codewords in $\mathcal{C}$ iff $s_{i,j}>1$.
Let $G_{\vec{s}}=((k),E)$ be the graph on vertex set $(k)$ be defined by $\{i,j\}\in E\iff s_{i,j}>1$.
Then
\begin{equation}\label{eq:graph}
\forb(1, \vec{s})=\omega(G_{\vec{s}}).
\end{equation} It is an easy exercise that the right hand side of (\ref{eq:ij-shatter})
 is an upper bound for this clique number in case of $n=1$.
The bound in (\ref{eq:ij-shatter}) follows from (\ref{eq:induction}) and (\ref{eq:graph})
 using induction and the well-known recurrence for the multinomial coefficients.
\hfill $\square$

\section{Forbidden configurations}
Another generalization or sharpening of Theorem~\ref{thm:sauer}
considers \emph{forbidden configurations}.
We say a (0,1)-matrix is \emph{simple} if there are no repeated rows.
Given a (0,1)-matrix $F$, we say a matrix $A$ has $F$ as a \emph{configuration} denoted $F\in A$,
 if there is a submatrix of $A$ which is a row and column permutation of $F$.
Let $|A|$ denote the number of rows of matrix $A$. We define
\begin{equation}\label{eq:forbdef}
\forb(n,F)=\max\{|A|\colon A \hbox{ is a simple 0-1 matrix
    without configuration }F \hbox{ of }n\hbox{ columns}\}.
\end{equation}
A simple (0,1)-matrix $A$ naturally corresponds to a set system $\mathcal{F}_A$
 taking the rows as characteristic vectors of subsets of $[n]$.
$\mathcal{F}_A$ shatters an $s$-set iff $A$ has the $2^s\times s$ configuration of all distinct rows
   of size $s$. \par
The concept of forbidden configurations can be extended for matrices of entries from $(k)$.
A $(k)$-matrix is \emph{simple} if there are no repeated rows. Given a $(k)$-matrix $F$,
 we say a matrix $A$ has $F$ as a \emph{configuration} denoted $F\in A$, if there is a
 submatrix of $A$ which is a row and column permutation of $F$.
Theorem~\ref{thm:multishatter} gives upper bound on $m$ for an $m\times n$ simple $(k)$-matrix that
 does not have the $k^s\times s$ configuration of all distinct rows of size $s$.\par
Definition (\ref{eq:forbdef}) of $\forb(n,F)$ can be applied to $(k)$-matrices, as well.
However, if polynomial upper bounds are desired, then more than one configurations must be
 forbidden simultaneously.
Let $\mathcal{F}=\{F_1,F_2,\ldots ,F_t\}$ be a collection of (not necessarily simple) $(k)$-matrices.
Let
  $$\forb(n,k,\mathcal{F})=\max\{m\colon A \hbox{ is }m\times
     n\hbox{ simple }(k)\hbox{-matrix and has no configuration }F\in\mathcal{F}\}.$$
In~\cite{fkms} it was proved that
$$
  \forb(n,k,\mathcal{F})=\Theta\left(n^{\binom{k}{2}}\right)  \enskip \hbox{ for }\enskip
 \mathcal{F}= \left\{\left(\begin{array}{cc}a&b\\b&a\end{array}\right)\colon 0\le a<b\le k-1\right\}.
   $$
Theorem~\ref{thm:ij-shatter} can also be reformulated in this language.
Let $F$ be a  (0,1)-matrix, then $F(i,j)$ denotes the $(i,j)$-matrix obtained
 from $F$ replacing 0's by $i$'s and 1's by $j$'s.
Let $K_s$ denote the  $2^s\times s$ (0,1)-matrix of all distinct rows of size $s$.
Theorem~\ref{thm:ij-shatter} gives  bounds for $\forb(n,k,\mathcal{F})$ where
   $\mathcal{F}= \left\{K_{s_{i,j}}(i,j)\colon  0\le i<j\le k-1\right\}$.
Here we prove a lower bound.
\begin{prop}\label{prop:lowerbound}
Let $F^{i,j}\colon 0\le i<j\le k-1$ be simple $(0,1)$-matrices such that
 none of them contains a constant column.
Then
\begin{equation}\label{eq:productlowerbd}
\forb(n,k,\{F^{i,j}(i,j)\colon 0\le i<j\le k-1\})
    \ge\prod_{ 0\le i<j\le k-1}\forb\left(\frac{n}{\binom{k}{2}},F^{i,j}\right).
\end{equation}
\end{prop}
\begin{proof}
We apply the  product construction introduced in \cite{AGS}.
Let $A^{i,j}$ be a simple $(0,1)$-matrix with $\frac{n}{\binom{k}{2}}$ columns and
 $\forb\left(\frac{n}{\binom{k}{2}},F^{i,j}\right)$ rows without configuration $F^{i,j}$.
Let
\begin{equation*}
A=A^{0,1}\times A^{0,2}\times\ldots\times A^{k-2,k-1}
\end{equation*}
be the matrix with $n$ columns and $ |A^{0,1}|\cdot |A^{0,2}|\cdot\ldots\cdot | A^{k-2,k-1}|$
 rows obtained by choosing one row from each of the matrices and putting them side by side in
  every possible way.
We claim that this product matrix $A$ avoids all configurations  $F^{i,j}\colon 0\le i<j\le k-1$.
Indeed, since each column of $F^{i,j}$ contains both symbols $i$ and $j$,
 columns of a configuration $F^{i,j}$ should come from columns of $A^{i,j}$ in the product.
Suppose $F^{i,j}$ has $p$ columns. Since  $F^{i,j}$ is simple and $A^{i,j}$ does not have
 configuration $F^{i,j}$, for each $p$-tuple of columns of $A^{i,j}$   there must be a row
 of  $F^{i,j}$ that is missing on those columns.
This will be missing in the product matrix, as well.
\end{proof}
Lower bound (\ref{eq:lowerbound}) follows by taking  $F^{i,j}=K_{s_{i,j}}\colon 0\le i<j\le k-1$
 and applying Theorem~\ref{thm:sauer}.

\section{Open problems}
There are more questions than answers known in  connection with $(i,j)$-shattering.
The principal problem is that Theorem~\ref{thm:ij-shatter} does not give sharp bounds,
 in contrast with Theorem~\ref{thm:sauer} and Theorem~\ref{thm:multishatter}.
We can give an exact bound only if most of the $s_{i,j}$'s are  ones.
\begin{prop}\label{prop:egyesek} Assume that $s_{i,j}=1$ if $i<j<k-1$. Then
\begin{equation*}
\forb(n,k,\vec{s})=\max_{\sum_{i=0}^{k-2}n_i=n}\prod_{i=0}^{k-2}
  \left(\binom{n_i}{s_{i,k-1}-1}+\binom{n_i}{s_{i,k-1}-2}+\ldots +\binom{n_i}{0}\right)
\end{equation*}
\end{prop}
\begin{proof} Suppose that $A$ is a $(k)$-matrix without configurations $K_{s_{i,j}}$.
$s_{i,j}=1$ means that symbols $i$ and $j$ cannot occur in the same column of $A$.
Thus columns of $A$ can be partitioned into $k-1$ parts, part $C_i$ containing only
 symbols $i$ and $k-1$ for $0\le i<k-1$.
The number of different projections onto column set  $C_i$ is
 $\binom{n_i}{s_{i,k-1}-1}+\binom{n_i}{s_{i,k-1}-2}+\ldots +\binom{n_i}{0}$ for $n_i=|C_i|$
 by Theorem~\ref{thm:sauer}.
Thus the maximum number of different rows of $A$ is at most
 $\prod_{i=0}^{k-2}\left(\binom{n_i}{s_{i,k-1}-1}+\binom{n_i}{s_{i,k-1}-2}+\ldots +\binom{n_i}{0}\right)$.
On the other hand the product construction (\ref{eq:productlowerbd}) provides a matching lower bound.
\end{proof}
It would be interesting to find exact bounds for other special cases, as well. \par
Another question whether containing no constant column or simplicity of the forbidden configurations
 is necessary condition in Proposition~\ref{prop:lowerbound}.
Also,  Proposition~\ref{prop:lowerbound} and Theorem~\ref{thm:ij-shatter} give asymptotically
 tight bounds if $\forb(n,F^{i,j})=\Theta(n^{s_{i,j}-1})$ where $s_{i,j}$ is the number of
 columns of $F^{i,j}$.
The question is that does Proposition~\ref{prop:lowerbound} give the correct order of magnitude
 of $\forb(n,k,\mathcal{F})$ for other lists $\mathcal{F}$  of forbidden configurations?

Since ${\rm VC}$-dimension has many of  applications
 in statistics, computer science and combinatorics, it seems likely that 
 bi-dimension can be applied there, too.

\end{document}